\documentclass[a4paper,12pt,centertags,final]{amsart}
\usepackage[T1]{fontenc}
\usepackage[utf8]{inputenc}
\usepackage{mathrsfs}
\usepackage{bbm}
\usepackage{amssymb}
\usepackage{eulervm}
\usepackage[a4paper,centering]{geometry}
\usepackage[pdfdisplaydoctitle,colorlinks,urlcolor=blue,linkcolor=blue,citecolor=blue]{hyperref}
\usepackage{mathtools}
\usepackage{mathscinet}
\usepackage{xspace}
\newtheorem{theorem}{Theorem}[section]
\newtheorem{lemma}[theorem]{Lemma}
\newtheorem{proposition}[theorem]{Proposition}

\theoremstyle{definition}

\theoremstyle{remark}
\newtheorem{remark}[theorem]{Remark}
\numberwithin{equation}{section}
\makeatletter
\newcounter{mr@ConstantsCounter}
\newcommand{\mconst}{
  \stepcounter{mr@ConstantsCounter}
  \ensuremath{c_\text{\themr@ConstantsCounter}}
}
\newcommand{\mcdef}[1]{
  \stepcounter{mr@ConstantsCounter}
  \protected@write\@auxout{}{
    \string\newlabel{#1}{
      {\themr@ConstantsCounter}{}{}{constant.\themr@ConstantsCounter}{}}
    }
    \hypertarget{constant.\themr@ConstantsCounter}{\mcref{#1}
  }
}
\makeatother
\newcommand{\mcref}[1]{
  \ensuremath{c_\text{\ref{#1}}}
}
\def\MRnum#1 #2\empty{#1}
\renewcommand{\MRhref}[2]{%
  \href{http://www.ams.org/mathscinet-getitem?mr=#1}{#2}
}
\renewcommand{\MR}[1]{%
  \relax\ifhmode\unskip\space\fi
  \MRhref{\MRnum#1\empty}{\texttt{\Tiny[MR\MRnum#1\empty]}}
}
\newcommand{\arxiv}[1]{\href{http://arxiv.org/abs/#1}{arXiv:#1}}
\DeclareMathOperator{\Div}{div}
\DeclareMathOperator{\Span}{span}
\DeclareMathOperator{\Tr}{Tr}
\newcommand{\e}{\operatorname{e}}       
\newcommand{\R}{\mathbf{R}}             
\newcommand{\scalar}[1]{\langle #1 \rangle}
\newcommand{\field}[1][F]{\mathscr{#1}} 
\newcommand{\Prob}{\mathbb{P}}          
\newcommand{\E}{\mathbb{E}}             
\newcommand{\loc}{{\textup{\tiny loc}}}
\newcommand{\uno}{\mathbbm{1}}          
\newcommand{\memo}[1]{ 
  \ensuremath{
    \framebox{\tiny\text{\kern-2pt\textsf{#1}}\kern-2pt}
  }
  \xspace
}
\DeclareMathOperator{\argmin}{arg\,min}
\newcommand{\cov}{\mathcal{S}}
\hypersetup{%
  pdftitle={Time regularity of the densities for the Navier--Stokes equations with noise},
  pdfauthor={M. Romito},
  pdfcreator={M. Romito}
}
\begin{document}
  \title[Time regularity of the densities]{Time regularity of the densities for the Navier--Stokes equations with noise}
  \author[M. Romito]{Marco Romito}
    \address{Dipartimento di Matematica, Università di Pisa, Largo Bruno Pontecorvo 5, I--56127 Pisa, Italia}
    \email{\href{mailto:romito@dm.unipi.it}{romito@dm.unipi.it}}
    \urladdr{\url{http://www.dm.unipi.it/pages/romito}}
  \date{}
  \begin{abstract}
    We prove that the density of the law of any finite dimensional projection
    of solutions of the Navier--Stokes equations with noise in dimension $3$
    is H\"older continuous in time with values in the natural space $L^1$.
    When considered with values in Besov spaces, H\"older continuity
    still holds. The H\"older exponents correspond, up to arbitrarily small
    corrections, to the expected diffusive scaling.
  \end{abstract}
\maketitle
\section{Introduction}

When dealing with a stochastic evolution PDE, the solution
depends not only on the time and space independent variables,
but also on the ``chance'' variable, that plays a completely
different role. Existence of a density for the distribution
of the solution is thus a form of regularity with respect
to the new variable. In infinite dimension there is no
canonical reference measure, therefore often existence of
densities is expected for finite dimensional functionals
of the solution.

This paper is a continuation of \cite{DebRom2014} and its
aim is to give an additional understanding of
the law of solutions of the Navier--Stokes equations driven
by noise in dimension three. More precisely, consider the
Navier--Stokes equations either on a smooth bounded domain
with zero Dirichlet boundary condition or on the 3D torus
with periodic boundary conditions and zero spatial mean,
\begin{equation}\label{e:nse}
  \begin{cases}
    \dot u + (u\cdot\nabla)u + \nabla p = \nu\Delta u + \dot\eta,\\
    \Div u = 0,
  \end{cases}
\end{equation}
where $u$ is the velocity, $p$ the pressure and $\nu$ the viscosity
of an incompressible fluid, and $\dot\eta$ is Gaussian noise, white
in time and coloured in space (see \cite{Fla2008} for a survey).
Existence of a density for finite
dimensional projections of the solution of \eqref{e:nse} and its
regularity in terms of Besov spaces was proved in \cite{DebRom2014}.
In this paper we prove that those densities are almost $\frac12$--%
H\"older continuous in time with values in $L^1$, as well as
with values in suitable Besov spaces defined on the finite
dimensional target space.

In a way, the results we obtain in this paper are not surprising.
After all we are dealing with a diffusion process and we already
know from \cite{DebRom2014} that the density has (in terms of Besov
regularity) almost one derivative. It is then expected that
the time regularity is of the order of (almost) half a derivative.
Likewise, if we look at the regularity of the derivative of order
$\alpha$, with $\alpha\in(0,1)$, a fair expectation
is that its time regularity is of order (almost) $\tfrac\alpha2$.
On the other hand, space regularity has been obtained
in a non--standard way by means of the method introduced
in \cite{DebRom2014}. As we will see time regularity requires
as well a non--trivial proof that mixes the method
of \cite{DebRom2014} with arguments based on the Girsanov
transformation. We believe that this adds value to the paper.

In a way, the problem at hand here can be considered
as part of a general attempt on proving existence and
regularity of densities of problems where, in principle,
Malliavin calculus is not immediately applicable. Here
the loss of regularity emerges due to infinite
dimension. To quickly understand that Malliavin
calculus is not directly applicable here, one can realize
that the equation that the Malliavin derivative of
the solution of \eqref{e:nse} should satisfy is
essentially the linearization (around $0$) of \eqref{e:nse}.
No good estimates on the linearization of \eqref{e:nse}
are available so far, as they could be used for
uniqueness as well.

The method we use has been developed in
\cite{DebRom2014}, starting from an idea of \cite{FouPri2010}
(see also \cite{Rom2013a} for a slightly more detailed account).
Later the same idea has been used in \cite{DebFou2013,Fou2012}.
An improvement of \cite{FouPri2010} in a different direction
has been given in \cite{BalCar2012}. Other attempts to handle
non--smooth problems are \cite{Dem2011}, and
\cite{KohTan2012,HayKohYuk2013,HayKohYuk2014}.
\section{Main results}

\subsection{Notations}

If $K$ is an Hilbert space, we denote by $\pi_F:K\to K$
the orthogonal projection of $K$ onto a subspace $F\subset K$, and by
$\Span[x_1,\dots,x_n]$ the subspace of $K$ generated by
its elements $x_1,\dots,x_n$. Given a linear operator
$\mathcal{Q}:K\to K'$, we denote by $\mathcal{Q}^\star$
its adjoint.
\subsubsection{Function spaces}

We recall the definition of Besov spaces. The general
definition is based on the Littlewood--Paley decomposition, but
it is not the best suited for our purposes. We shall use
an alternative equivalent definition (see \cite{Tri1983,Tri1992})
in terms of differences. Given $f:\R^d\to\R$, define
\[
  \begin{gathered}
    (\Delta_h^1f)(x)
      = f(x+h)-f(x),\\
    (\Delta_h^nf)(x)
      = \Delta_h^1(\Delta_h^{n-1}f)(x)
      = \sum_{j=0}^n (-1)^{n-j}\binom{n}{j} f(x+jh),
  \end{gathered}
\]
and, for $s>0$, $1\leq p\leq\infty$, $1\leq q<\infty$,
\[
  [f]_{B_{p,q}^s}
    = \Bigl(\int_{\{|h|\leq 1\}}\frac{\|\Delta_h^n f\|_{L^p}^q}{|h|^{sq}}
          \frac{dh}{|h|^d}\Bigr)^{\frac1q},
\]
and for $q=\infty$,
\[
  [f]_{B_{p,\infty}^s}
    = \sup_{|h|\leq 1}\frac{\|\Delta_h^n f\|_{L^p}}{|h|^s},
\]
where $n$ is any integer larger than $s$.
Given $s>0$, $1\leq p\leq\infty$ and $1\leq q\leq\infty$,
define
\[
  B_{p,q}^s(\R^d)
    = \{f: \|f\|_{L^p} + [f]_{B_{p,q}^s}<\infty\}.
\]
This is a Banach space when endowed with the norm
$\|f\|_{B_{p,q}^s} := \|f\|_{L^p} + [f]_{B_{p,q}^s}$.

When in particular $p=q=\infty$ and $s\in(0,1)$,
the Besov space $B_{\infty,\infty}^s(\R^d)$
coincides with the H\"older space $C^s_b(\R^d)$,
and in that case we will denote by $\|\cdot\|_{C^s_b}$
and $[\cdot]_{C^s_b}$ the corresponding norm and
semi--norm.
\subsubsection{Navier Stokes framework}

Let $H$ be the standard space of square summable divergence free
vector fields, defined as the closure of divergence free smooth
vector fields satisfying the boundary condition (either zero
Dirichlet or periodic, with zero spatial mean in the latter case),
with inner product $\scalar{\cdot,\cdot}_H$ and norm $\|\cdot\|_H$.
Define likewise $V$ as the closure of the same space of test
functions with respect to the $H^1$ norm.

Let $\Pi_L$ be the Leray projector, $A=-\Pi_L\Delta$ the Stokes
operator, and denote by $(\lambda_k)_{k\geq1}$ and $(e_k)_{k\geq1}$
the eigenvalues and the corresponding orthonormal basis of eigenvectors
of $A$. Define the bi--linear operator $B:V\times V\to V'$
as $B(u,v) = \Pi_L\left( u\cdot\nabla v\right)$, $u,v\in V$,
and recall that $\scalar{u_1, B(u_2,u_3)}=-\scalar{u_3, B(u_2,u_1)}$.
We refer to Temam \cite{Tem1995} for a detailed account of all
the above definitions.

The noise $\dot\eta=\cov\dot W$ in \eqref{e:nse} is coloured
in space by a covariance operator $\cov^\star\cov\in\mathscr{L}(H)$,
where $W$ is a cylindrical Wiener process (see \cite{DapZab1992}
for further details). We assume that $\cov^\star\cov$ is trace--class
and we denote by $\sigma^2 = \Tr(\cov^\star\cov)$
its trace. Finally, consider the sequence $(\sigma_k^2)_{k\geq1}$
of eigenvalues of $\cov^\star\cov$, and let $(q_k)_{k\geq1}$ be the
orthonormal basis in $H$ of eigenvectors of $\cov^\star\cov$.
\subsection{Galerkin approximations}

With the above notations, we can recast problem \eqref{e:nse}
as an abstract stochastic equation,
\begin{equation}\label{e:nseabs}
  du + (\nu Au + B(u))\,dt = \cov\,dW,
\end{equation}
with initial condition $u(0)=x\in H$. It is well--known \cite{Fla2008}
that for every $x\in H$ there exist a martingale solution of this equation,
that is a filtered probability space $(\widetilde\Omega,
\widetilde{\field},\widetilde{\Prob},\{\widetilde{\field}_t\}_{t\geq 0})$,
a cylindrical Wiener process $\widetilde W$ and a process $u$ with trajectories
in  $C([0,\infty);D(A)')\cap L^\infty_\loc([0,\infty),H)\cap
L^2_\loc([0,\infty);V)$
adapted to $(\widetilde{\field}_t)_{t\geq 0}$ such that the above
equation is satisfied with $\widetilde W$ replacing $W$.

We will consider in particular solutions of \eqref{e:nse} obtained
as limits of Galerkin approximations.
Given an integer $N\geq 1$, denote by $H_N$ the sub--space
$H_N = \Span[e_1,\dots,e_N]$ and denote by $\pi_N = \pi_{H_N}$
the projection onto $H_N$. It is standard (see for instance
\cite{Fla2008}) to verify that the problem
\begin{equation}\label{e:galerkin}
  du^N + \bigl(\nu Au^N + B^N(u^N))\,dt = \pi_N\cov\,dW,
\end{equation}
where $B^N(\cdot) = \pi_N B(\pi_N\cdot)$, admits a unique strong
solution for every initial condition $x^N\in H_N$. Moreover,
\begin{equation}\label{e:Gbound}
  \E\Bigl[\sup_{[0,T]}\|u^N\|_H^p\Bigr]
    \leq c_p(1+\|x^N\|_H^p),
\end{equation}
for every $p\geq1$ and $T>0$, where $c_p$ depends
only on $p$, $T$ and the trace of $\cov\cov^\star$.

If $x\in H$, $x^N = \pi_N x$ and $\Prob^N_x$ is the distribution
of the solution of the problem above with initial condition $x^N$,
then any limit point of $(\Prob^N_x)_{N\geq1}$ is a solution
of the martingale problem associated to \eqref{e:nse} with initial
condition $x$.
\begin{remark}
  In general, there is nothing special with the basis provided by
  the eigenvectors of the Stokes operator and our results would
  work when applied to Galerkin approximations generated by
  any (smooth enough) orthonormal basis of $H$. The crucial
  assumption is that the solution is a limit point of finite
  dimensional approximations. Some of the results concerning
  densities (but not those in this paper) can be generalized
  to any martingale weak solution of \eqref{e:nseabs}, see
  \cite{Rom2014a}.
\end{remark}
\subsection{Assumptions on the covariance}

Given a finite dimensional subspace $F$ of $H$,
we assume the following non degeneracy condition on
the covariance,
\begin{equation}\label{e:hpgirsanov}
  \cov x = f
    \qquad\text{has a solution for every }f\in F,
\end{equation}
The condition above is stronger than the condition
\begin{equation}\label{e:hpbesov}
  \pi_F\cov\cov^\star\pi_F
    \quad\text{is a non--singular matrix},
\end{equation}
used in \cite{DebRom2014} to prove bounds on the
Besov norm of the density. It is not clear if our results
here may be true under the weaker assumption
\eqref{e:hpbesov} (see Remark~\ref{r:besovmaybe} though).

Indeed, for our method --- that works through finite
dimensional approximations, it is convenient to
assume a slightly stronger version of \eqref{e:hpgirsanov},
namely that
\begin{equation}\label{e:hpgirsanov2}
  \pi_N\cov x = f
    \text{\quad has a solution for every }f\in F,
\end{equation}
for $N$ large enough.
\subsection{Continuity in time of the density}

Our first main result is that densities of finite dimensional
projections of solutions of \eqref{e:nseabs} are
continuous (actually H\"older with exponent almost $\tfrac12$)
with respect to time 
with values in the natural space $L^1$ of densities.
\begin{theorem}\label{t:main1}
  Let $F$ be a finite dimensional subspace of 
  $D(A)$ generated by a finite set of eigenvalues
  of the Stokes operator, and assume \eqref{e:hpgirsanov2}.
  
  Given $\alpha\in(0,1)$, there is $\mcdef{cc:main1}>0$
  such that if $x\in H$ and $u$ is a weak solution of
  \eqref{e:nseabs} with initial condition $x$ that
  is a limit point of Galerkin approximations,
  if $f(\cdot;x)$ is the density with respect to the
  Lebesgue measure on $F$ of the random variable
  $\pi_F u(\cdot)$, then
  \[
    \|f(t;x) - f(s;x)\|_{L^1(F)}
      \leq \mcref{cc:main1}(1+s\vee t)^{\frac{1-\alpha}2}
        \|f(s\wedge t)\|_{B^\alpha_{1,\infty}}(1 + \|x\|_H^2)^2
        |t-s|^{\frac\alpha2},
  \]
  for every $s,t>0$.
\end{theorem}
The theorem above follows immediately from
Proposition~\ref{p:elleuno} and lower semicontinuity.
Notice that the term $\|f(s\wedge t)\|_{B^\alpha_{1,\infty}}$
is singular when $s\wedge t$ approaches $0$ (see
Lemma~\ref{l:timedep}).

By trading time--continuity with space--time continuity,
we can obtain an estimate similar to the one given in
the above theorem for the Besov norm of the density.
\begin{theorem}\label{t:main2}
  Let $F$ be a finite dimensional subspace of 
  $D(A)$ generated by a finite set of eigenvalues
  of the Stokes operator, and assume \eqref{e:hpgirsanov2}.

  Given $\alpha,\beta\in(0,1)$ with
  $\alpha+\beta<1$, there is $\mcdef{cc:main2}>0$
  such that if $x\in H$ and $u$ is a weak solution of
  \eqref{e:nseabs} with initial condition $x$ that
  is a limit point of Galerkin approximations,
  if $f(\cdot;x)$ is the density with respect to the
  Lebesgue measure on $F$ of the random variable
  $\pi_F u(\cdot)$, then
  \[
    \|f(t;x) - f(s;x)\|_{B^{\alpha}_{1,\infty}}
      \leq\mcref{cc:main2}|t-s|^{\frac\beta2},
  \]
  for every $s,t>0$, where
  \[
    \mcref{cc:main2}
      \approx (1+s\vee t)^{\frac{1-\beta}2}(1+\|x\|_H^2)^3
        \bigl([f(t)]_{B^{1-\delta}_{1,\infty}}
         + [f(s)]_{B^{1-\delta}_{1,\infty}}\bigr),
  \]
  and $\delta<1-(\alpha+\beta)$.
\end{theorem}
The proof of this theorem is given by means of
Proposition~\ref{p:besov}. A crucial tool in
the proof of both theorems is Girsanov's
transformation. This explains why we need
the slightly stronger assumption \eqref{e:hpgirsanov}
rather than the assumption \eqref{e:hpbesov}
used in \cite{DebRom2014}. Girsanov's change
of measure is used to perform a sort of fractional
integration by parts and move the tiny regularity
from space to time (see Lemma~\ref{l:Phiincrement}).
\section{The estimate in \texorpdfstring{$L^1$}{L1}}

This section is devoted to the proof of the H\"older estimate
of the density with values in $L^1$. A classical way is to
derive first some space regularity and then use it to prove
the time regularity. In a way, this is also the bulk of our
method, although due to the low regularity we have at hand
(see Lemma~\ref{l:timedep}), this can be done only after
a suitable simplification. The main tool we use here is
the Girsanov transformation and the logarithmic moments
of the Girsanov density. The version of the Girsanov theorem
we use follows from \cite[Chapter 7]{LipShi2001}.
The main result of this section is as follows.
\begin{proposition}\label{p:elleuno}
  Let $F$ be a finite dimensional subspace of 
  $D(A)$ generated by a finite set of eigenvalues
  of the Stokes operator, and assume \eqref{e:hpgirsanov2}.
  Given $\alpha\in(0,1)$, there is $\mcdef{cc:elleuno}>0$
  such that if $x\in H$, $N$ is large enough (that $F\subset H_N$)
  and $u^N$ is a solution of \eqref{e:galerkin} with initial condition
  $\pi_N x$, if $f_N(\cdot;x)$ is the density with respect to the
  Lebesgue measure on $F$ of $\pi_F u^N(\cdot)$, then
  \[
    \|f_N(t) - f_N(s)\|_{L^1(F)}
      \leq \mcref{cc:elleuno}(1+s\vee t)^{\frac{1-\alpha}2}
        \|f_N(s\wedge t)\|_{B^\alpha_{1,\infty}}(1 + \|x\|_H^2)^2
        |t-s|^{\frac\alpha2},
  \]
  for every $s,t>0$.
\end{proposition}
In the rest of the section we will drop, for simplicity
and to make the notation less cumbersome, the index $N$.
It is granted though that we work with solutions of the
Galerkin system \eqref{e:galerkin}.
\subsection{The Girsanov equivalence}\label{s:girsanov}

Let us assume now \eqref{e:hpgirsanov2} and consider the
following two stochastic equations on $H_N$
\[
  \begin{gathered}
    du + (\nu A u + \pi_N B(u))\,dt
      = \pi_N\cov\,dW,\\
    dv + (\pi_N-\pi_F)(\nu A v + B(v))\,dt
      = \pi_N\cov\,dW.\\
  \end{gathered}
\]
It is easy to see that both equations have a unique
strong solution for every initial condition in $H_N$.
In view of the application of the Girsanov transformation,
assume $u(0) = v(0)\in H_N$.
\subsubsection{The Moore--Penrose pseudo--inverse}

Given a linear bounded operator $\cov:H\to H$ and a finite
dimensional subspace $F\subset H$ such that $\cov x=f$
has at least one solution for every $f\in F$, define
\[
  \cov^+f
    = \argmin\{\|x\|_H: x\in H\text{ and }\cov x = f\}.
\]
It is elementary to check that the pseudo--inverse
$\cov^+:F\to H$ is well defined and is a linear
bounded operator, since given $f$ the minima $x$
are characterized by $\scalar{x,y-x}_H\geq0$
for every $y\in H$ such that $\cov y=f$.
In particular $\cov\cov^+f = f$ and,
if Assumption \eqref{e:hpgirsanov2} holds for $\cov$,
$(\pi_N\cov)^+ = \cov^+$.
\subsubsection{Reduction by the Girsanov transformation}

Fix for the rest of the section $T>0$. If $w\in C([0,T];H_N)$,
set
\[
  \tau_n(w)
    = \inf\Bigl\{t\leq T: \int_0^T \|\cov^+\pi_F\bigl(\nu Aw + B(w)\bigr)\|_H^2\,ds\geq n\Bigr\},
\]
and $\tau_n(w)=T$ if the above set is empty,
and $\chi_t^n(w)=\uno_{\{\tau_n(w)\geq t\}}$. By \eqref{e:Gbound}
$\tau_n(u)<\infty$ almost surely. Similar computations yield
that also $\tau_n(v)<\infty$ almost surely.

Let $v^n$ be the solution of
\[
  \begin{multlined}[.9\linewidth]
  v^n(t)
    = v(t\wedge\tau_n(v))
        - \int_0^t(1 - \chi_s^n(v))\pi_N(\nu Av^n+B(v^n))\,ds + {}\\
        + \int_0^t(1 - \chi_s^n(v))\pi_N\cov\,dW,
  \end{multlined}
\]
then $v^n(t) = v(t)$ on $\{\tau_n(v)\geq t\}$,
$\tau^n_t(v) = \tau^n_t(v^n)$, and
$v^n(t)\to v(t)$ almost surely. More
precisely, $v^n(t) = v(t)$ for $n$ large enough
($\omega$--wise), therefore $\phi(v^n(t))\to\phi(v(t))$
almost surely for any bounded measurable $\phi$.

Moreover, since
\[
  v(t\wedge\tau_n(v))
    = v(0)
      - \int_0^t\chi_s^n(v)(\pi_N - \pi_F)(\nu Av+B(v))\,ds
      + \int_0^t\chi_s^n(v)\pi_N\cov\,dW,
\]
it follows that
\[
  \begin{multlined}[.9\linewidth]
    v^n(t)
     = v(0)
       - \int_0^t(\nu Av^n + \pi_NB(v^n))\,ds + {}\\
       + \int_0^t\pi_N\cov\,dW
       + \int_0^t\chi_s^n(v^n)\pi_F (\nu Av^n+B(v^n))\,ds.
  \end{multlined}
\]
By the Girsanov theorem the process
\[
  \begin{multlined}[.9\linewidth]
    G_t^n
      = \exp\Bigl(\int_0^t\chi^n_s(v^n)\cov^+\pi_F(\nu Av^n+B(v^n))\,dW_s + {}\\
        -\frac12\int_0^t\chi^n_s(v^n)\|\cov^+\pi_F(\nu Av^n + B(v^n))\|_H^2\,ds\Bigr)
  \end{multlined}
\]
is a martingale and the law of $u$ on $[0,T]$ with respect to
the original probability measure $\Prob$ is equal to the law
of $v^n$ on $[0,T]$ with respect to the new probability
measure $G_T^n\Prob$.
\subsection{Increments of the Girsanov density}

In this section we estimate the time increments of the Girsanov
density. This provides half of the proof of Proposition~\ref{p:elleuno}.
\begin{lemma}\label{l:logG}
  There is $\mcdef{cc:logG}>0$ such that for every
  $0\leq s\leq t\leq T$ and every $n\geq1$,
  \[
    \E\Bigl[G_t^n\Big|\log\frac{G_n^t}{G_n^s}\Big|\Bigr]
      \leq\mcref{cc:logG}(t-s)^\frac12(1 + \|u(0)\|_H^2)^2.
  \]
\end{lemma}
\begin{proof}
  By changing back the probability measure, since on the interval $[0,t]$
  $u$ under $\Prob$ has the same law of $v^n$ under $G^n_t\Prob$,
  \[
    \begin{aligned}
      \E\Bigl[G_t^n\Bigl|\log\frac{G_t^n}{G_s^n}\Bigr|\Bigr]
        &= \E\Bigl[\Bigl|\log\frac{G_t^n(u)}{G_s^n(u)}\Bigr|\Bigr]\\
        &\leq \E\Bigl[2\Big|\int_s^t\chi_r^n(u)\cov^+\pi_F(\nu Au+B(u))\,dW_r\Big|\Bigr]\\
        &\quad  + \E\Bigl[\int_s^t\chi_r^n(u)\|\cov^+\pi_F(\nu Au+B(u))\|_H^2\,dr\Bigr]\\
        &\leq \mcref{cc:logG}(t-s)^\frac12(1 + \|u(0)\|_H^2)^2,
    \end{aligned}
  \]
  where we have used the Burkholder-Davis-Gundy inequality and
  \eqref{e:Gbound}.
\end{proof}
\begin{lemma}\label{l:Gincrement}
  There is $\mcdef{cc:girsanov}>0$ such that for every
  $0\leq s\leq t\leq T$ and $n\geq1$,
  \[
    |\E[(G_t^n - G_s^n)X]|
      \leq \mcref{cc:girsanov}\|X\|_\infty(1 + \|u(0)\|_H^2)^2(t-s)^\frac12,
  \]
  where $X$ is any real bounded random variable.
\end{lemma}
\begin{proof}
  Without loss of generality, we assume $\|X\|_\infty\leq 1$.
  Fix $0\leq s\leq t\leq T$ and notice that, since $G_t^n$ is
  a martingale, $\E[G_t^n-G_s^n]=0$, hence
  \[
    \E[(G_t^n-G_s^n)\uno_{\{G_t^n\geq G_s^n\}}]
      = \E[(G_s^n-G_t^n)\uno_{\{G_s^n\geq G_t^n\}}].
  \]
  Thus
  \[
    \begin{aligned}
      |\E[(G_t^n - G_s^n)X]|
        &= |\E[(G_t^n-G_s^n)X\uno_{\{G_t^n\geq G_s^n\}}]
            + \E[(G_t^n-G_s^n)X\uno_{\{G_s^n\geq G_n^n\}}]|\\
        &\leq \E[(G_t^n-G_s^n)\uno_{\{G_t^n\geq G_s^n\}}]
           + \E[(G_s^n - G_t^n)X\uno_{\{G_s^n\geq G_t^n\}}]\\
        &= 2\E[(G_t^n-G_s^n)\uno_{\{G_t^n\geq G_s^n\}}]\\
        &= 2\E\Bigl[G_s\Bigl(\e^{\log\frac{G_t^n}{G_s^n}}-1\Bigr)
             \uno_{\{G_t^n\geq G_s^n\}}\Bigr],
    \end{aligned}
  \]
  and, by using the elementary inequality
  $\e^x-1\leq(1\wedge|x|)\e^x$, $x\in\R$,
  \[
    \begin{aligned}
      |\E[(G_t^n - G_s^n)X]|
        &\leq 2\E\Bigl[G_s\Bigl(\e^{\log\frac{G_t^n}{G_s^n}}-1\Bigr)
             \uno_{\{G_t^n\geq G_s^n\}}\Bigr]\\
        &\leq 2\E\Bigl[G_s\Bigl(1\wedge\log\frac{G_t^n}{G_s^n}\Bigr)
             \frac{G_t^n}{G_s^n} \uno_{\{G_t^n\geq G_s^n\}}\Bigr]\\
        &\leq 2\E\Bigl[G_t^n\Bigl(1\wedge\Bigl|\log\frac{G_t^n}{G_s^n}\Bigr|\Bigr)\Bigr]\\
        &\leq 2\E\Bigl[G_t^n\Bigl|\log\frac{G_t^n}{G_s^n}\Bigr|\Bigr].
    \end{aligned}
  \]
  Finally, the conclusion of the lemma follows by
  Lemma~\ref{l:logG}.
\end{proof}
\subsection{Proof of Proposition~\ref{p:elleuno}}

We recall an elementary inequality, its proof is
straightforward calculus: for every $x,y\geq0$
and $\epsilon>0$,
\begin{equation}\label{e:elementary}
  xy
    \leq\epsilon\e^{\frac{y}{\epsilon}} + \epsilon x\log x.
\end{equation}
\begin{lemma}\label{l:tobm}
  For every $\epsilon>0$, every $s,t\in[0,T]$, every $n\geq1$
  and every bounded measurable $\phi:F\to\R$,
  \[
    |\E[G_s^n\bigl(\phi(\pi_Fv^n(t)) - \phi(\pi_F v(t))\bigr)]|
      \leq \epsilon\|\phi\|_\infty\bigl(
        \mcref{cc:logG}\sqrt{T}(1+\|u(0)\|_H^2)^2
        + \e^{\frac2\epsilon}\Prob[\tau_n(v)<t]\bigr).
  \]
\end{lemma}
\begin{proof}
  Fix $\epsilon>0$ and assume for simplicity $\|\phi\|_\infty\leq 1$.
  We know that $v^n(t) = v(t)$ on $\tau_n(v)\geq t$, hence
  \[
    \E[G_s^n\bigl(\phi(\pi_Fv^n(t)) - \phi(\pi_F v(t))\bigr)]
      = \E[G_s^n\bigl(\phi(\pi_Fv^n(t)) - \phi(\pi_F v(t))\bigr)
        \uno_{\{\tau_n(v)<t\}}].
  \]
  By the inequality \eqref{e:elementary} above, applied to $x=G_s^n$
  and $y=\frac1\epsilon(\phi(\pi_Fv^n(t)) - \phi(\pi_F v(t)))$,
  \[
    \begin{multlined}[.9\linewidth]
      \E[G_s^n\bigl(\phi(\pi_Fv^n(t)) - \phi(\pi_F v(t))\bigr)
          \uno_{\{\tau_n(v)<t\}}]\leq\\
        \leq \epsilon\E[G_s^n\log G_s^n]
          + \epsilon\E[\e^{\phi(\pi_Fv^n(t)) - \phi(\pi_F v(t))}
          \uno_{\{\tau_n(v)<t\}}]\leq\\
        \leq \epsilon\E[G_s^n\log G_s^n]
          + \epsilon\e^\frac{2}{\epsilon}\Prob[\tau_n(v)<t].
    \end{multlined}
  \]
  The statement of the lemma now follows by Lemma~\ref{l:logG}.
\end{proof}
Let $U_\phi$ be the solution of the heat equation
\begin{equation}\label{e:heat}
  \partial_t U_\phi
    = \frac12\Tr(\pi_F\cov(\pi_F\cov)^\star D^2U_\phi),
\end{equation}
with initial condition $\phi$. This is well defined, smooth
and a linear transformation of the standard heat equation
due again to assumption \eqref{e:hpbesov}. 
\begin{lemma}\label{l:rep}
  For every $0\leq s\leq t\leq T$, $n\geq1$ and $\phi:F\to\R$
  bounded measurable,
  \[
    \E[G_s^n\phi(\pi_F v(t))]
      = \E[G_s^n U_\phi(t-s,\pi_F v(s))].
  \]
\end{lemma}
\begin{proof}
  Set $\beta(t) = \pi_Fv(t)$, then by assumption \eqref{e:hpbesov}
  $\beta(t) = \pi_Fu(0) + \int_0^t \pi_F\cov\,dW$ is
  a $d$--dimensional Brownian motion started at $\pi_Fu(0)$.
  By the Markov property,
  \[
    \E[G_s^n\phi(\pi_Fv(t))]
      = \E\bigl[G_s^n\E[\phi(\beta(t))|\field_s]\bigr]
      = \E[G_s^n U_\phi(t-s,\beta_s)].\qedhere
  \]
\end{proof}
\begin{lemma}\label{l:Phiincrement}
  There is $\mcdef{cc:heat}>0$ such that for every $0\leq s\leq t\leq T$,
  every $n\geq1$, every bounded measurable $\phi:F\to\R$, and
  every $\alpha\in(0,1)$,
  \[
    \begin{aligned}
      \E[G_s^n\bigl(\phi(\pi_Fv^n(t)) - \phi(\pi_F v^n(s))\bigr)]
        &\leq \mcref{cc:heat}\|\phi\|_\infty\bigl(
          [f(s)]_{B^\alpha_{1,\infty}}(t-s)^{\frac\alpha2}\\
        &\quad  + \epsilon\sqrt{T}(1+\|u(0)\|_H^2)^2
          + \epsilon\e^{\frac2\epsilon}\Prob[\tau_n(v)<t]\bigr).
    \end{aligned}
  \]
\end{lemma}
\begin{proof}
  Let $s,t,n,\phi$ as in the statement of the lemma and assume
  for simplicity $\|\phi\|_\infty\leq 1$. We have
  \[
    \begin{aligned}
    \E[G_s^n\bigl(\phi(\pi_Fv^n(t)) - \phi(\pi_F v^n(s))\bigr)]
      &= \underbrace{\E[G_s^n\bigl(\phi(\pi_Fv^n(t)) - 
          U_\phi(t-s,\pi_F v^n(s))\bigr)]}_{\memo{\Tiny a}}\\
      &\quad + \underbrace{\E[G_s^n\bigl(U_\phi(t-s,\pi_Fv^n(s))
          - \phi(\pi_F v^n(s))\bigr)]}_{\memo{\Tiny b}}.
    \end{aligned}
  \]
  For the first term we use Lemma~\ref{l:rep},
  Lemma~\ref{l:tobm} twice, and
  $\|U_\phi\|_\infty\leq\|\phi\|_\infty$,
  \[
    \begin{aligned}
    \memo{a}
      &= \E[G_s^n\bigl(\phi(\pi_Fv^n(t)) - \phi(\pi_F v(t))\bigr)]
        + \E[G_s^n\bigl(\phi(\pi_Fv(t)) - U_\phi(t-s,\pi_F v(s))\bigr)]\\
      &\quad  + \E[G_s^n\bigl(U_\phi(t-s,\pi_Fv(s)) - U_\phi(t-s,\pi_F v^n(s))\bigr)]\\
      &\leq 2\epsilon\bigl(\mcref{cc:logG}\sqrt{T}(1+\|u(0)\|_H^2)^2
        + \e^{\frac2\epsilon}\Prob[\tau_n(v)<t]\bigr).
    \end{aligned}
  \]
  For the second term, we change back the probability measure, since
  on the interval $[0,s]$ $u$ under $\Prob$ has the same law of $v^n$
  under $G^n_s\Prob$,
  \[
    \begin{aligned}
      \memo{b}
        &= \E[\bigl(U_\phi(t-s,\pi_Fu(s)) - \phi(\pi_F u(s))\bigr)]\\
        &= \int_{\R^d}(U_\phi(t-s,y) - \phi(y))f(s,y)\,dy\\
        &= \int_{\R^d}(\hat\E[\phi(y+\hat B_{t-s})] - \phi(y))f(s,y)\,dy\\
        &= \hat\E\Bigl[\int_{\R^d}\phi(y)(f(s,y-\hat B_{t-s})-f(s,y))\,dy\Bigr]\\
        &\leq \hat\E[\|f(s,\cdot-\hat B_{t-s})-f(s,\cdot)\|_{L^1(\R^d)}]\\
        &\leq [f(s)]_{B^\alpha_{1,\infty}}\hat\E[|\hat B_{t-s}|^\alpha]\\
        &\leq\mconst[f(s)]_{B^\alpha_{1,\infty}}(t-s)^{\frac\alpha2},
    \end{aligned}
  \]
  where $\alpha\in(0,1)$, $f(t,\cdot)$ (or more precisely $f_N(t,\cdot)$,
  but again we drop the superscript for simplicity) is the density
  of $\pi_Fu(t)$, and where $(\hat B_t)_{t\geq0}$ is an independent
  $F$--valued Brownian motion with (spatial) covariance
  $\pi_F\cov(\pi_F\cov)^\star$ introduced to represent the solutions
  of \eqref{e:heat}.
\end{proof}
We finally have all the ingredients to complete the proof of
Proposition~\ref{p:elleuno}.
\begin{proof}[Proof of Proposition~\ref{p:elleuno}]
  Let $0\leq s\leq t$.
  By duality, it sufficient to estimate the following quantity
  for each bounded measurable $\phi:F\to\R$ with $\|\phi\|_\infty\leq1$.
  For every $n\geq1$, by the Girsanov transformation detailed
  in Section~\ref{s:girsanov},
  \[
    \begin{aligned}
      \int_F \phi(y)(f(t,y) - f(s,y))\,dy
        &= \E[\phi(\pi_F u(t)) - \phi(\pi_F u(s))]\\
        &= \E[G_t^n\bigl(\phi(\pi_F v^n(t)) - \phi(\pi_F v^n(s))\bigr)]\\
        &= \E[G_t^n\phi(\pi_F v^n(t)) - G_s^n\phi(\pi_F v^n(s))]\\
        &= \underbrace{\E[(G_t^n-G_s^n)\phi(\pi_F v^n(t))]}_{\memo{\Tiny 1}}\\
        &\quad + \underbrace{\E[G_s^n\bigl(\phi(\pi_F v^n(t)) - \phi(\pi_F v^n(s))\bigr)]}_{\memo{\Tiny 2}}.
    \end{aligned}
  \]
  The first term is estimated through Lemma~\ref{l:Gincrement},
  \[
    \memo{1}
      \leq \mcref{cc:girsanov}(1 + \|x\|_H^2)^2(t-s)^\frac12,
  \]
  the second term through Lemma~\ref{l:Phiincrement}, for every $\epsilon>0$,
  \[
    \memo{2}
      \leq \mcref{cc:heat}\bigl([f(s)]_{B^\alpha_{1,\infty}}(t-s)^{\frac\alpha2}
        + \epsilon\sqrt{t}(1+\|x\|_H^2)^2
        + \epsilon\e^{\frac2\epsilon}\Prob[\tau_n(v)<t]\bigr),
  \]
  so that in conclusion
  \[
    \begin{multlined}[.9\linewidth]
    \Bigl|\int_F \phi(y)(f(t,y) - f(s,y))\,dy\Bigr|
      \leq \mcref{cc:girsanov}(1 + \|x\|_H^2)^2(t-s)^\frac12 + {}\\
        + \mcref{cc:heat}\bigl([f(s)]_{B^\alpha_{1,\infty}}(t-s)^{\frac\alpha2}
        + \epsilon\sqrt{t}(1+\|x\|_H^2)^2
        + \epsilon\e^{\frac2\epsilon}\Prob[\tau_n(v)<t]\bigr),
    \end{multlined}
  \]
   and by taking first the limit as $n\uparrow\infty$, so that
   $\Prob[\tau_n(v)<t]\downarrow0$, and then as $\epsilon\downarrow0$,
   the statement of the proposition follows.
\end{proof}
\section{The estimate in the Besov seminorm}

In this section we prove Theorem~\ref{t:main2}. To this end
we use together the machinery on Girsanov's theorem introduced
in the previous section and the technique based on Besov spaces
introduced in \cite{DebRom2014}.
\subsection{A smoothing lemma}

The technique introduced in \cite{DebRom2014} is based
on a duality estimate that provides a quantitative integration
by parts. Since we are dealing with regularity properties of
low order, we will use Besov spaces to measure it.
The following lemma is implicitly given in \cite{DebRom2014},
we state it here explicitly and give a complete proof.
\begin{lemma}[smoothing lemma]\label{l:smoothing}
  If $\mu$ is a finite measure
  on $\R^d$ and there are an integer $m\geq1$,
  two real numbers $s>0$, $\gamma\in(0,1)$,
  with $\gamma<s<m$, and a constant $K>0$
  such that for every $\phi\in C^\gamma_b(\R^d)$ and
  $h\in\R^d$,
  \[
    \Bigl|\int_{\R^d} \Delta_h^m\phi(x)\,\mu(dx)\Bigr|
      \leq K|h|^s\|\phi\|_{C_b^\gamma},
  \]
  then $\mu$ has a density $f_\mu$ with respect to the
  Lebesgue measure on $\R^d$. Moreover, for every $r<s-\gamma$
  there exists $\mcdef{cc:smoothing}>0$ such that
  \begin{equation}\label{e:smoothing}
    \|f_\mu\|_{B^r_{1,\infty}}
      \leq\mcref{cc:smoothing}(\mu(\R^d) + K).
  \end{equation}
\end{lemma}
\begin{proof}
  Fix a smooth function $\phi$.
  Let $(\varphi_\epsilon)_{\epsilon>0}$ be a smoothing kernel,
  namely $\varphi_\epsilon = \epsilon^{-d}\varphi(x/\epsilon)$,
  with $\varphi\in C^\infty_c(\R^d)$, $0\leq\varphi\leq1$, and
  $\int_{\R^d}\varphi(x)\,dx = 1$. Let $f_\epsilon=\varphi_\epsilon
  \star\mu$, then easy computations show that $f_\epsilon\geq0$,
  $\int_{\R^d}f_\epsilon(x)\,dx =\mu(\R^d)$ and that
  \[
    \Bigl|\int_{\R^d}\Delta_h^m\phi(x) f_\epsilon(x)\,dx\Bigr|
      = \Bigl|\int\varphi_\epsilon(x)\Bigl(
        \int_{\R^d}\Delta_h^m\phi(x-y)\,\mu(dy)\Bigl)\,dx\Bigr|
      \leq K|h|^s\|\phi\|_{C^\gamma_b}.
  \]
  On the other hand, by a discrete integration by parts,
  \begin{equation}\label{e:dibp}
    \int_{\R^d}\Delta_h^m\phi(x) f_\epsilon(x)\,dx
      = \int_{\R^d}\Delta_{-h}^mf_\epsilon(x) \phi(x)\,dx.
  \end{equation}
  Set $g_\epsilon=(I-\Delta_d)^{-\beta/2}f_\epsilon$,
  and $\psi = (I-\Delta_d)^{\beta/2}\phi$, where $\Delta_d$
  is the $d$--dimensional Laplace operator and $\beta>\gamma$.
  We have by \cite[Theorem 10.1]{AroSmi1961} that
  $\|g_\epsilon\|_{L^1}\leq\mconst\|f_\epsilon\|_{L^1}$. Moreover,
  by \cite[Theorem 2.5.7,Remark 2.2.2/3]{Tri1983}),
  we know that
  $C^\gamma_b(\R^d) = B^\gamma_{\infty,\infty}(\R^d)$,
  and by \cite[Theorem 2.3.8]{Tri1983} we know that
  $(I - \Delta_d)^{-\beta/2}$ is a continuous operator
  from $B_{\infty,\infty}^{\gamma-\beta}(\R^d)$ to
  $B_{\infty,\infty}^\gamma(\R^d)$.
  Hence, by \eqref{e:dibp} it follows that
  \[
    \int_{\R^d}\Delta_h^mg_\epsilon(x)\psi(x)\,dx
      = \int_{\R^d}\Delta_h^mf_\epsilon(x)\phi(x)\,dx
      \leq K|h|^s\|\phi\|_{C^\gamma_b}
      \leq\mcdef{cc:smoothing2}K|h|^s\|\psi\|_{B^{\gamma-\beta}_{\infty,\infty}}
  \]
  Notice that by \cite[Theorem 2.11.2]{Tri1983},
  $B_{\infty,\infty}^{\gamma-\beta}(\R^d)$ is the dual
  of $B_{1,1}^{\beta-\gamma}(\R^d)$, moreover
  $B_{1,1}^{\beta-\gamma}(\R^d)\hookrightarrow L^1(\R^d)$ by definition,
  since $\beta>\gamma$, therefore
  $L^\infty(\R^d)\hookrightarrow B^{\gamma-\beta}_{\infty,\infty}$.
  By duality, $\|\Delta_h^mg_\epsilon\|_{L^1}\leq\mcref{cc:smoothing2}K|h|^s$,
  hence $\|g_\epsilon\|_{B^s_{1,\infty}}\leq\mconst(K+\mu(\R^d))$.
  Again since $(I - \Delta_d)^{\beta/2}$ maps continuously
  $B_{\infty,\infty}^s(\R^d)$ into
  $B_{\infty,\infty}^{s-\beta}(\R^d)$, it finally follows
  that $\|f_\epsilon\|_{B^{s-\beta}_{1,\infty}}\leq\mconst\|g_\epsilon\|_{B^s_{1,\infty}}$
  for every $\beta>\gamma$.

  By Sobolev's embeddings and \cite[formula 2.2.2/(18)]{Tri1983},
  we have for every $r<s-\beta$ and $1\leq p\leq d/(d-r)$ that
  $B_{\smash{1,\infty}}^{s-\beta}(\R^d)\hookrightarrow
  B_{1,1}^r(\R^d) = W^{r,1}(\R^d)\subset L^p(\R^d)$.
  In particular, $(f_\epsilon)_{\epsilon>0}$ is uniformly integrable
  in $L^1(\R^d)$, therefore there is $f_\mu$ such that $\mu = f_\mu\,dx$
  and $(f_\epsilon)_{\epsilon>0}$ converges weakly in $L^1(\R^d)$
  to $f_\mu$. By semi--continuity, \eqref{e:smoothing} holds for
  every $r<s-\gamma$.
\end{proof}
\subsection{The Besov estimate}

Let $x\in H$ and consider a solution $u$ of \eqref{e:nseabs}
that is a limit point of Galerkin approximations. All our estimates
will pass to the limit and so it is not restrictive to work
on the solution $u^N$ of \eqref{e:galerkin} with initial condition
$u^N(0) = \pi_N x$.

Given $t>0$ and $\epsilon\in(0,t)$, let
$\chi_{t,\epsilon}=\uno_{[0,t-\epsilon]}$ be the indicator
function of the interval $[0,t-\epsilon]$, and let
$u_\epsilon^N$ be the solution of
\begin{equation}\label{e:besovkill}
  d u_\epsilon^N
      + (\pi_N - \pi_F)\bigl(\nu A u_\epsilon^N + B(u_\epsilon^N)\bigr)\,dt
      + \chi_{t,\epsilon}\pi_F B(u_\epsilon^N)\,dt
    = \pi_N \cov\,dW,
\end{equation}
that is $u_\epsilon^N = u^N$ up to time $t-\epsilon$,
and $\tilde u=\pi_F u_\epsilon^N$ satisfies
for $r\in [t-\epsilon,t]$, 
\[
  \tilde u(r)
    = \pi_Fu^N(t-\epsilon)
      + \pi_F\cov (W_r-W_{t-\epsilon}).
\]
Due to assumption \eqref{e:hpbesov}, $\tilde u(r)$ is
a $d$-dimensional Brownian motion (where $d$ is the
dimension of $F$) with spatial covariance matrix
$\pi_F\cov\cov^\star\pi_F$.
\begin{proposition}\label{p:besov}
  Let $F$ be a finite dimensional subspace of 
  $D(A)$ generated by a finite set of eigenvalues
  of the Stokes operator, and assume \eqref{e:hpgirsanov2}.

  Given $\alpha,\beta\in(0,1)$ with
  $\alpha+\beta<1$, there is $\mcdef{cc:besov}>0$
  such that if $x\in H$, if $N$ is large enough
  (that $F\subset H_N$) and $u^N$ is a weak solution of
  \eqref{e:galerkin} with initial condition $\pi_Nx$,
  if $f_N(\cdot;x)$ is the density with respect to the
  Lebesgue measure on $F$ of the random variable
  $\pi_F u^N(\cdot)$, then
  \[
    \|f_N(t) - f_N(s)\|_{B^{\alpha}_{1,\infty}}
      \leq\mcref{cc:besov}|t-s|^{\frac\beta2},
  \]
  for every $s,t>0$, where
  \[
    \mcref{cc:besov}
      \approx (1+s\vee t)^{\frac{1-\beta}2}(1+\|x\|_H^2)^3
        \bigl([f_N(t)]_{B^{1-\delta}_{1,\infty}}
         + [f_N(s)]_{B^{1-\delta}_{1,\infty}}\bigr),
  \]
  and $\delta<1-(\alpha+\beta)$.
\end{proposition}
The following lemma summarizes the result of \cite{DebRom2014},
adding the explicit dependence of the Besov norm of the
density in terms of time, which is needed for the
evaluation of the inequality in the previous proposition.
\begin{lemma}\label{l:timedep}
  Let $F$ be a finite dimensional subspace of 
  $D(A)$ generated by a finite set of eigenvalues
  of the Stokes operator, and assume \eqref{e:hpbesov}.
  For every $t>0$ and $x\in H$, the projection $\pi_F u(t)$
  has a density $f_F(t)$ with respect to the Lebesgue
  measure on $F$, where $u$ is any solution of \eqref{e:nseabs},
  with initial condition $x$, which is a limit point of
  the spectral Galerkin approximations.

  Moreover, for every $\alpha\in(0,1)$,
  $f_F(t)\in B_{1,\infty}^\alpha(\R^d)$ and for every (small)
  $\epsilon>0$, there exists
  $\mcdef{cc:besovtd}=\mcref{cc:besovtd}(\alpha,\epsilon)>0$
  such that
  \[
    \|f_F(t)\|_{B^\alpha_{1,\infty}}
      \leq\frac{\mcref{cc:besovtd}}{(1\wedge t)^{\alpha+\epsilon}}
        (1 + \|x\|_H^2)^{\alpha+\epsilon}.
  \]
\end{lemma}
\begin{proof}
  Given a finite dimensional space $F$ as in the statement,
  fix $t>0$, and let $\gamma\in(0,1)$, $\phi\in C_b^\gamma$,
  and $h\in F$, with $|h|\leq 1$.
  For $m\geq1$, consider two cases. If $|h|^{2n/(2\gamma+n)}<t$,
  then we use the same estimate in \cite{DebRom2014} to get
  \[
    \bigl|\E[\Delta_h^m\phi(\pi_F u(t))]\bigr|
      \leq\mconst(1+\|x\|_H^2)^\gamma\|\phi\|_{C^\gamma_b}
        |h|^{\frac{2n\gamma}{2\gamma+n}}.
  \]
  If on the other hand $t\leq |h|^{2n/(2\gamma+n)}$, we
  introduce the process $u_\epsilon$ as above, but with
  $\epsilon=t$. As in \cite{DebRom2014},
  \[
    \E[\Delta_h^m\phi(\pi_F u(t))]
      = \E[\Delta_h^m\phi(\pi_F u_\epsilon(t))]
        + \E[\Delta_h^m\phi(\pi_F u(t)) - \Delta_h^m\phi(\pi_F u_\epsilon(t))]
  \]
  and
  \[
    \bigl|\E[\Delta_h^m\phi(\pi_F u(t)) - \Delta_h^m\phi(\pi_F u_\epsilon(t))]\bigr|
      \leq \mconst(1 + \|x\|_H^2)^\gamma\|\phi\|_{C^\gamma_b}t^\gamma.
  \]
  For the probabilistic error we use the fact that $u_\epsilon(t)$
  is Gaussian, hence
  \[
    \bigl|\E[\Delta_h^m\phi(\pi_F u_\epsilon(t))]\bigr|
      \leq \mconst\|\phi\|_\infty
        \Bigl(\frac{|h|}{\sqrt{t}}\Bigr)^{\frac{2n\gamma}{2\gamma+n}}
  \]
  In conclusion, from both cases we finally have
  \[
    \bigl|\E[\Delta_h^m\phi(\pi_F u(t))]\bigr|
      \leq \mconst(1+\|x\|_H^2)^\gamma\|\phi\|_{C^\gamma_b}
        |h|^{\frac{2n\gamma}{2\gamma+n}}
        (1\wedge t)^{-\frac{n\gamma}{2\gamma+n}}.
  \]
  Given $\alpha$, suitable choices of $n$ and $\gamma$ yield
  the final result.
\end{proof}
\begin{lemma}\label{l:brownian}
  Let $\beta_r = \pi_F\cov W_r$, $r\geq0$. There is
  $\mcdef{cc:brownian}>0$ such that
  \[
    |\E[\Delta_h^n\phi(a + \beta_r) - \Delta_h^n\phi(a + \beta_s)]|
      \leq\frac{\mcref{cc:brownian}}{r\vee s}\|\phi\|_\infty
        \Bigl(\frac{|h|}{\sqrt{r\wedge s}}\Bigr)^n|r-s|,
  \]
  for every $a\in F$, $n\geq1$, $\phi\in C_c^\infty(F)$, $h\in F$
  with $|h|_F\leq 1$, and $r,s\geq0$.
\end{lemma}
\begin{proof}
  By assumption \eqref{e:hpbesov}, $\beta$ is a $d$--dimensional
  Brownian motion with covariance matrix $\pi_F\cov\cov^\star\pi_F$.
  If $\mathcal{Q}$ is a $d\times d$ matrix such that
  $\pi_F\cov\cov^\star\pi_F = \mathcal{Q}\mathcal{Q}^\star$,
  then $\beta_r = \mathcal{Q}B_r$, where $B_r$ is a standard
  $d$--dimensional Brownian motion. The position
  $\psi(x) = \phi(a+\mathcal{Q}x)$ reduces the statement to the
  same for a standard Brownian motion.
  The latter is a straightforward estimate.
\end{proof}
In the rest of the section we will drop, for simplicity
and to make the notation less cumbersome, the index $N$.
It is granted though that we work with solutions of the
Galerkin system \eqref{e:galerkin}.
\begin{lemma}\label{l:numgirsanov}
  Assume \eqref{e:hpgirsanov} and let $v$ be the process introduced
  in Section~\ref{s:girsanov}. Given $\gamma\in(0,1)$, there exists
  $\mcdef{cc:numgirsanov}>0$ such that for every $0<s\leq t$ and every
  bounded measurable $\psi:F\to\R$,
  \begin{multline}\label{e:numgirsanov}
    \big|\E[\psi(\pi_Fu(t)) - \psi(\pi_Fu(s))]
        - \E[\psi(\pi_Fv(t)) - \psi(\pi_Fv(s))]\big|\leq\\
      \leq \mcref{cc:numgirsanov}(1+\|u(0)\|_H^2)^2\log(2+\|u(0)\|_H^2)
        \sqrt{t}(-\log(\tfrac12\wedge t))[\psi]_{C^\gamma_b}
        (t-s)^{\frac\gamma2}.
  \end{multline}
\end{lemma}
\begin{proof}
  We work in the framework introduced in Section~\ref{s:girsanov}.
  Let us denote, for brevity, the left--hand side of
  \eqref{e:numgirsanov} by \memo{num}. We have that
  \[
    \memo{num}
      = \E[G_t^n\bigl(\psi(\pi_F v^n(t)) - \psi(\pi_F v^n(s))\bigr)]
        - \E[\psi(\pi_Fv(t)) - \psi(\pi_Fv(s))]
  \]
  First we  notice that we can replace $v^n$ by $v$ in the
  above formula, up to an error that converges to $0$ as $n\to\infty$.
  Indeed, by Lemma~\ref{l:tobm}, for every $\delta>0$,
  \[
    \begin{multlined}[.9\linewidth]
      \big|\E[G_t^n\psi(\pi_F v^n(t))] - \E[G_t^n\psi(\pi_F v(t))]\big|\leq\\
        \leq \delta\|\psi\|_\infty\bigl(
          \mconst\sqrt{t}(1 + \|x\|_H^2)^2
          + \e^{\frac2\delta}\Prob[\tau_n(v)<t],
    \end{multlined}
  \]
  and likewise at time $s$, where $u(0)=x$.
  After replacing $v^n$ by $v$ we will obtain an estimate that
  is uniform in $n$. By taking first the limit as $n\to\infty$
  and then as $\delta\downarrow0$, the lemma will be proved.

  After this preliminary observation, we see that
  \[
    \memo{num}
      \approx \E[(G_t^n-1)\bigl(\psi(\beta_t) - \psi(\beta_s)\bigr)],
  \]
  where, as in Lemma~\ref{l:rep}, $\beta_t=\pi_Fv(t)$ is a
  d--dimensional Brownian motion. By \eqref{e:elementary}, for
  given $a,b>0$ that will be given later,
  \[
    \memo{num}
      \lesssim ab\E\Bigl[\frac{|G_t^n-1|}{a}\log\frac{|G_t^n-1|}{a}\Bigr]
        + ab\E\Bigl[\exp\Bigl(\frac{|\psi(\beta_t)
          - \psi(\beta_s)|}{b}\Bigr)\Bigr]
        = ab\,\memo{1} + ab\,\memo{2}.
  \]
  Notice that
  \[
    |\psi(\beta_t) - \psi(\beta_s)|
      \leq [\psi]_{C^\gamma_b}|\beta_t-\beta_s|^\gamma
      = [\psi]_{C^\gamma_b}|t-s|^{\frac\gamma2}|Z|^\gamma,
  \]
  where $Z$ is a Gaussian random variable whose distribution does not depend
  on $s,t$. If we choose $b=[\psi]_{C^\gamma_b}|t-s|^{\frac\gamma2}$,
  then $\memo{2}\leq\E[\exp(|Z|^\gamma)]\leq\mconst$.

  For the first term \memo{1} we see that
  \[
    \memo{1}
      = \frac1a\E[|G^n_t-1|\log|G^n_t-1|]
        - \frac{\log a}{a}\E[|G^n_t-1|].
  \]
  The same argument of Lemma~\ref{l:Gincrement}
  (here $G^n_t-1=G^n_t-G^n_0$) yields
  \[
    \E[|G^n_t-1|]
      \leq\mconst(1 + \|x\|_H^2)^2\sqrt{t}
  \]
  Moreover, by Lemma~\ref{l:logG},
  \[
    \begin{multlined}[.95\linewidth]
      \E[|G^n_t-1|\log|G^n_t-1|] =\\
        = \E[(G^n_t-1)\log(G^n_t-1)\uno_{\{G^n_t\geq2\}}]
          + \underbrace{\E[|G^n_t-1|\log|G^n_t-1|\uno_{\{G^n_t\leq2\}}]}_{\leq0}\leq\\
        \leq \E[G^n_t\log G^n_t\uno_{\{G^n_t\geq2\}}]
        \leq \E[G^n_t|\log G^n_t|]
        \leq\mconst\sqrt{t}(1 + \|x\|_H^2)^2.
    \end{multlined}
  \]
  In conclusion
  \[
    \memo{1}
      \leq \mconst\frac{\sqrt{t}}{a}(1 + \|x\|_H^2)^2(1+|\log a|).
  \]
  The choice
  \[
    a
      \approx(1 + \|x\|_H^2)^2\log(2 + \|x\|_H^2)
        \sqrt{t}(-\log(\tfrac12\wedge t)),
  \]
  yields $\memo{1}\leq\mconst$.
\end{proof}
\begin{proof}[Proof of Proposition~\ref{p:besov}]
  Denote by $f$ the density of $\pi_F u$.
  Let $s\leq t$, $\phi\in C^\infty_c(\R^d)$, $h\in F$
  with $|h|_F\leq 1$, and fix the parameters
  $\gamma,\delta\in(0,1)$, $\epsilon>0$,
  $n\geq 3$ that will be chosen along the proof.

  Assume that $t-s\leq|h|^2$,
  then by a discrete integration by parts,
  \[
    \begin{aligned}
      \int_F \phi\Delta_h^n(f(t)-f(s))\,dx
        &= \int_F (f(t) - f(s))\Delta_{-h}^n\phi\,dx\\
        &= \E[\Delta_{-h}^n\phi(\pi_Fu(t)) - \Delta_{-h}^n\phi(\pi_Fu(s))]\\
        &= \underbrace{\E[\Delta_{-h}^n\phi(\pi_Fu(t)) - \Delta_{-h}^n\phi(\pi_Fu_\epsilon(t))]}_{\memo{num${}_t$}}\\
        &\quad  + \underbrace{\E[\Delta_{-h}^n\phi(\pi_Fu_\epsilon(t)) - \Delta_{-h}^n\phi(\pi_Fu_\epsilon(s))]}_{\memo{prob}}\\
        &\quad + \underbrace{\E[\Delta_{-h}^n\phi(\pi_Fu_\epsilon(s)) - \Delta_{-h}^n\phi(\pi_Fu(s))]}_{\memo{num${}_s$}},
    \end{aligned}
  \]
  where $u_\epsilon$ has been defined in \eqref{e:besovkill}.
    
  To estimate $\memo{prob}$, we first point out that we will choose
  $\epsilon$ so that $t-s\leq\frac\epsilon2$. Notice that
  \[
    \memo{prob}
      = \E\bigl[\E[\Delta_{-h}^n\phi(\pi_Fu_\epsilon(t))
        - \Delta_{-h}^n\phi(\pi_Fu_\epsilon(s))]\,|\,\field_{t-\epsilon}]\bigr],
  \]
  and that, given $\field_{t-\epsilon}$, $\pi_Fu_{N,\epsilon}(r)$
  has the same law of $\pi_Fu(t-\epsilon)+\beta_{r-t+\epsilon}$,
  where $\beta$ is the process of Lemma~\ref{l:brownian}.
  Hence, by Lemma~\ref{l:brownian}, and since $t-s\leq\frac\epsilon2$,
  \begin{equation}\label{e:prob}
    \begin{aligned}
      \memo{prob}
        &= \E\bigl[\E[\Delta_{-h}^n\phi(\pi_Fu(t-\epsilon)+\beta_\epsilon)
             - \Delta_{-h}^n\phi(\pi_Fu(t-\epsilon)+\beta_{s-t+\epsilon})]
             |\field_{t-\epsilon}\bigr]\\
        &\leq\frac{\mconst}{\epsilon^{1+\frac{n}2}}\|\phi\|_\infty|h|^n|t-s|.
    \end{aligned}
  \end{equation}
  Let \memo{num} = \memo{num${}_s$} + \memo{num${}_t$}, then
  by conditioning
  \[
    \begin{multlined}[.9\linewidth]
      \memo{num}
        = \E\bigl[\E[\Delta_{-h}^n\phi(\pi_Fu(t))
          - \Delta_{-h}^n\phi(\pi_Fu(s))\,|\field_{t-\epsilon}]\bigr] + {}\\
          - \E\bigl[\E[\Delta_{-h}^n\phi(\pi_Fu_\epsilon(t))
          - \Delta_{-h}^n\phi(\pi_Fu_\epsilon(s))
          \,|\field_{t-\epsilon}]\bigr].
    \end{multlined}
  \]
  We use the Markov property and Lemma~\ref{l:numgirsanov}
  with times $s-t+\epsilon$ and $\epsilon$, and
  $\psi= \Delta_{-h}^n\phi$ to get
  \begin{equation}\label{e:num}
    \begin{aligned}
      \memo{num}
        &\leq \mcref{cc:numgirsanov}
           \E[(1+\|u(t-\epsilon)\|_H^2)^2\log(2+\|u(t-\epsilon)\|_H^2)]
           \sqrt{\epsilon}(-\log\epsilon)[\phi]_{C^\gamma_b}
           (t-s)^{\frac\gamma2}\\
        &\leq\mconst(1+\|x\|_H^2)^3\sqrt{\epsilon}(-\log\epsilon)
           [\phi]_{C^\gamma_b} (t-s)^{\frac\gamma2}.
    \end{aligned}
  \end{equation}
  
  In conclusion \eqref{e:prob} and \eqref{e:num} yield
  \[
    \Big|\int_F \phi\Delta_h^n(f(t)-f(s))\,dx\Big|
      \leq \mconst(1+\|x\|_H^2)^3\|\phi\|_{C^\gamma_b}
        (t-s)^{\frac\gamma2}
        \Bigl(\epsilon^{\frac12(1-\delta)}
          + \frac{|h|^n}{\epsilon^{1+\frac{n}2}}\Bigr),
  \]
  where $\delta\in(0,1)$ has been introduced to get rid of the
  log correction and simplify computations. By optimizing in $\epsilon$
  we choose $\epsilon^{\frac12(n+3-\delta)}\sim|h|^n$, that
  is $\epsilon\sim|h|^{\frac{2n}{n+3-\delta}}$ (the
  exponent of $|h|$ is smaller than $2$, hence
  $(t-s)\lesssim\epsilon$ and $(t-s)$ can be made smaller
  than $\frac\epsilon2$ by a suitable constant). We finally
  have
  \begin{equation}\label{e:smaller}
    \Bigl|\int_F \phi\Delta_h^n(f(t)-f(s))\,dx\Bigr|
      \leq \mconst(1+\|x\|_H^2)^3\|\phi\|_{C^\gamma_b}
           (t-s)^{\frac\gamma2}|h|^{E_n},
  \end{equation}
  with $E_n=\frac{n(1-\delta)}{n+3-\delta}\uparrow1-\delta$.

  If on the other hand $t-s\geq|h|^2$, by integrating by parts once
  in the discrete variable,
  \[
    \begin{aligned}
      \Bigl|\int_F \phi\Delta_h^n(f(t)-f(s))\,dx\Bigr|
        &= \Bigl|\int_F (\Delta_{-h}\phi)\Delta_h^{n-1}(f(t)-f(s))\,dx\Bigr|\\
        &\leq \|\Delta_h^{n-1}(f(t) - f(s))\|_{L^1(F)}\|\Delta_{-h}\phi\|_\infty\\
        &\leq [f(t) - f(s)]_{B^{1-\delta}_{1,\infty}}
          [\phi]_{C^\gamma_b}|h|^{1-\delta+\gamma}\\
        &\leq \bigl([f(t)]_{B^{1-\delta}_{1,\infty}}
           + [f(s)]_{B^{1-\delta}_{1,\infty}}\bigr)
           \|\phi\|_{C^\gamma_b}|h|^{1-\delta+\gamma}.	
    \end{aligned}
  \]

  Since $|h|^2\leq(t-s)$, $|h|\leq 1$, and $E_n\leq 1-\delta$,
  $|h|^{1-\delta+\gamma}\leq (t-s)^{\frac\gamma2}|h|^{E_n}$,
  and we finally get
  \begin{equation}\label{e:larger}
    \Bigl|\int_F \phi\Delta_h^n(f(t)-f(s))\,dx\Bigr|
      \leq \bigl([f(t)]_{B^{1-\delta}_{1,\infty}}
         + [f(s)]_{B^{1-\delta}_{1,\infty}}\bigr)
         \|\phi\|_{C^\gamma_b}(t-s)^{\frac\gamma2}
         |h|^{E_n}.
  \end{equation}
  We have all the ingredients to conclude the proof. Let
  $\beta\in(0,1)$ and $\alpha<1-\beta$, and choose
  $\gamma=\beta$. Choose $\delta$ small enough and
  $n$ large enough that $E_n\geq\alpha+\beta$. Then
  Proposition~\ref{p:elleuno} and the same arguments of
  Lemma~\ref{l:smoothing} yield that
  \[
    \|f(t) - f(s)\|_{B^{\alpha}_{1,\infty}}
      \leq\mcref{cc:besov}(t-s)^{\frac\beta2},
  \]
  where \mcref{cc:besov} is the sum of the contribution
  from Proposition~\ref{p:elleuno} and the maximum between
  the contributions from \eqref{e:smaller} and \eqref{e:larger}.
\end{proof}
\begin{remark}\label{r:besovmaybe}
  A worse estimate can be obtained if one want to avoid
  Girsanov's transformation and assumption \eqref{e:hpgirsanov},
  and rely only on assumption \eqref{e:hpbesov} (at least when
  giving an estimate of the Besov seminorm). Indeed, instead
  of using Lemma~\ref{l:numgirsanov}, we estimate the \memo{num}
  terms in two different ways, to take into account both the control
  by $|t-s|$ and by $\epsilon$. On the one hand, to estimate
  \memo{num${}_s$} and \memo{num${}_t$},
  notice that if $r\in[t-\epsilon,t]$,
  \begin{equation}\label{e:num1}
    |\memo{num${}_r$}|
      \leq 2^n[\phi]_{C^\gamma_b}\E[\|\pi_Fu(r) - \pi_Fu_\epsilon(r)\|_H^\gamma]\\
      \leq \mconst(1 + \|x\|_H^2)^\gamma\epsilon^\gamma[\phi]_{C^\gamma_b},
  \end{equation}
  since
  \[
    \pi_Fu(r) - \pi_F u_\epsilon(r)
      = -\int_{t-\epsilon}^r\bigl(\nu\pi_FAu(\rho) + 
          \pi_F B(u(\rho))\bigr)\,d\rho,
  \]
  hence, by \eqref{e:Gbound},
  \[
    \E[\|\pi_Fu(r) - \pi_Fu_\epsilon(r)\|_H]
      \leq \mconst\int_{t-\epsilon}^r (1+\E[\|u(\rho)\|_H^2)\,d\rho]
      \leq \mconst\epsilon(1 + \|x\|_H^2).
  \]
  On the other hand,
  \begin{equation}\label{e:num2}
    \memo{num${}_s$} + \memo{num${}_t$}
      \leq \mconst[\phi]_{C^\gamma_b}(1+t^{\frac\gamma2})
        (1 + \|x\|_H^2)^\gamma(t-s)^{\frac\gamma2},
  \end{equation}
  since
  \[
    \pi_F(u(t)-u(s))
      = -\pi_F\int_s^t(\nu Au(r)+B(u(r)))\,dr
        + \pi_F\cov(W_t-W_s),
  \]
  hence by \eqref{e:Gbound} and standard estimates on the
  Wiener process,
  \[
    \E[\|\pi_Fu(t) - \pi_Fu(s)\|_H]
      \leq\mconst(1+\sqrt{t})(1+\|x\|_H^2)\sqrt{t-s},
  \]
  (and likewise but simpler for the increment of $u_\epsilon$).

  In conclusion, using \eqref{e:prob}, \eqref{e:num1},
  and \eqref{e:num2}, for every $\lambda\in(0,1)$,
  \[
    \begin{aligned}
      \Bigl|\int_F \phi\Delta_h^n(f(t)-f(s))\,dx\Bigr|
        &\leq \mconst\|\phi\|_{C^\gamma_b}(1+\|x\|_H^2)^\gamma
           (1+t)^{\frac12\gamma(1-\lambda)}\cdot\\
        &\quad\cdot\Bigl(\epsilon^{\lambda\gamma}
          (t-s)^{\frac12\gamma(1-\lambda)}
          + \frac{|h|^n}{\epsilon^{\frac{n+2}2}}(t-s)\Bigr).
    \end{aligned}
  \]
  Optimize in $\epsilon$ and choose $\epsilon\sim
  |h|^{\frac{2n}{n+\gamma(1+\lambda)}}$ (the exponent
  is smaller than $2$, hence $(t-s)\lesssim\epsilon$
  and can be made smaller than $\frac\epsilon2$ by a
  suitable constant) to get
  \begin{equation}\label{e:smaller2}
    \begin{multlined}[.9\linewidth]
      \Bigl|\int_F \phi\Delta_h^n(f(t)-f(s))\,dx\Bigr|\leq\\
        \leq \mconst\|\phi\|_{C^\gamma_b}(1+\|x\|_H^2)^\gamma
             (1+t)^{\frac12\gamma(1-\lambda)}
             (t-s)^{\frac12\gamma(1-\lambda)}
             |h|^{\frac{2n\gamma\lambda}{n+\gamma(1+\lambda)}}.
    \end{multlined}
  \end{equation}
  The case $t-s\geq|h|^2$ yields, as in the previous proof,
  \begin{equation}\label{e:larger2}
    \begin{multlined}[.8\linewidth]
      \Bigl|\int_F \phi\Delta_h^n(f(t)-f(s))\,dx\Bigr|\leq\\
        \leq \bigl([f(t)]_{B^\gamma_{1,\infty}}
            + [f(s)]_{B^\gamma_{1,\infty}}\bigr)
          \|\phi\|_{C^\gamma_b}(t-s)^{\frac12\gamma(1-\lambda)}
          |h|^{\frac{2n\gamma\lambda}{n+\gamma(1+\lambda)}}.
    \end{multlined}
  \end{equation}

  The estimate in $L^1$ would be as follows: given
  $a,b>0$ with $a+2b<1$, choose $\gamma\in(a+2b,1)$
  and $\lambda=1-b/\gamma$, so that $a<(2\lambda-1)\gamma$,
  hence there is $n$ large enough such that
  \[
    \frac{2n\lambda\gamma}{n+\gamma(1+\lambda)}-\gamma
      \leq a,
  \]
  and by using \eqref{e:smaller2} and \eqref{e:larger2}, the same arguments
  of Lemma~\ref{l:smoothing} and Proposition~\ref{p:elleuno} yield that
  \[
    \|f(t) - f(s)\|_{B^a_{1,\infty}}
      \leq \mconst(s,t,a,b)(t-s)^{\frac12b}.
  \]
\end{remark}
\end{document}